\documentclass{amsart}
\usepackage{graphicx} 
\numberwithin{equation}{section}
\usepackage{amsmath}
\usepackage{amssymb}
\usepackage{amsthm}
\usepackage{color}
\usepackage{enumerate}
\theoremstyle{definition}
\newtheorem{thm}{Theorem}[section]
\newtheorem{prop}[thm]{Proposition}
\newtheorem{rem}[thm]{Remark}
\newtheorem{lemma}[thm]{Lemma}

\newtheorem{example}[thm]{Example}

\newcommand{\edge}{\mathcal{A}}
\newcommand{\p}{\mathcal{I}}
\newcommand{\h}{H}
\newcommand{\return}{\rho}
\newcommand{\indupot}{\tilde\phi}
\newcommand{\indumap}{\tilde f}
\newcommand{\edgesindu}{\tilde{\mathcal{A}}}
\newcommand{\inducoding}{\tilde\Sigma_B}
\newcommand{\indupote}{\tilde\phi}
\newcommand{\frat}{A}
\newcommand{\exponent}{\gamma(f)}
\newcommand{\edgepart}{\tilde{E}}
\newcommand{\inducingdomain}{\mathcal{D}}
\newcommand{\codingmap}{\tilde{\pi}}
\newcommand{\measure}{\tilde{\mu}}
\newcommand{\codingmeasure}{\tilde{\mu}'}
\newcommand{\shift}{\tilde{\sigma}}
\newcommand{\indupressure}{\tilde{p}}
\newcommand{\maps}{\mathcal{S}}
\newcommand{\conv}{\text{Conv}(\{\delta_{x_i}\}_{i\in\p})}
\newcommand{\indulimit}{\tilde{\Lambda}}
\newcommand{\LB}{\text{LB}(q)}
\newcommand{\dimension}{\delta}
\newcommand{\palpha}{p_\alpha}

\title[Thermodynamic formalism and multifractal analysis]{Thermodynamic formalism and multifractal analysis of Birkhoff averages for non-uniformly expanding interval maps with finitely many branches}
\author{Yuya Arima}
\date{\today}
\address{Graduate School of Mathematics, Nagoya University,
Furocho, Chikusaku, Nagoya, 464-8602, JAPAN} 
\email{yuya.arima.c0@math.nagoya-u.ac.jp}

\begin{document}

\begin{abstract}
    In this paper, we perform a multifractal analysis of Birkhoff averages for interval maps with finitely many branches and parabolic fixed points. Using the thermodynamic approach, we strengthen the results of Johansson et al. \cite{JJTOPnonuniformly} on the conditional variational principle for the multifractal spectra of Birkhoff averages. To do this, we develop several refined properties of the thermodynamic formalism for non-uniformly expanding interval maps. 
\end{abstract}

\maketitle

\section{Introduction}

Let $f:\Lambda\rightarrow \Lambda$ be a dynamical system on a compact subset $\Lambda$ of $[0,1]$ and let $\phi$ be a continuous potential on $\Lambda$.
The Birkhoff average of $\phi$ at $x\in \Lambda$ is defined by the time average
\[
\lim_{n\to\infty}\frac{1}{n}\sum_{i=0}^{n-1}\phi(f^i(x))
\]
whenever the limit exists. 
Birkhoff averages provide a way to characterize the dynamical system $f$.
For a given $f$-invariant ergodic Borel probability measure $\mu$ on $\Lambda$,
Birkhoff's ergodic theorem yields that  for $\mu$-a.e. $x\in \Lambda$ the Birkhoff average of $\phi$ at $x$ converges to the space average $\int \phi d\mu$.
Thus, for $\alpha\neq \int \phi d\mu$ the set $B(\alpha)$ of points that the Birkhoff average of $\phi$ converges to $\alpha$ is negligible with respect to $\mu$. However, there is still a possibility that $B(\alpha)$  might be a large set from another point of view. This raises the following natural questions: What are the typical or exceptional Birkhoff averages? How large is the set $B(\alpha)$? To answer these questions we define the Birkhoff spectrum $\alpha\mapsto b(\alpha)$, where $b(\alpha)$ denotes  the Hausdorff dimension with respect to the Euclidean metric on $\mathbb{R}$ of the set $B(\alpha)$ and study its properties. 
We refer the reader to  the books Pesin \cite{Pesinbook} and Barreira \cite{Barreirabook} for an introduction to the subject of dynamical systems and the dimension theory.
In the uniformly hyperbolic case, the Birkhoff spectrum for a H\"older continuous potential has been well studied  by Barreira and Saussol \cite{BarreiraSaussol}. In the non-uniformly hyperbolic case, Lyapunov spectrum, that is, the Birkhoff spectrum for the geometric potential has also been well studied (see \cite{GelfertRams}, \cite{kessebohmer2004multifractaltobeappdated}, \cite{Nakaishi}, \cite{PollicottWeiss} and also \cite{Iommilyapunov}).
For a general continuous potential, Johansson et al. \cite{JJTOPnonuniformly} established the conditional variational principle (see also \cite{MixedJT} for results on mixed multifractal spectra of Birkhoff averages). 
Later, Climenhaga \cite{simultaneous} strengthened their results for non-uniformly expanding maps with a single parabolic fixed point.  
It then follows from these results that the Birkhoff spectrum is continuous and monotone on a certain domain. However, the following natural questions remain open: Is it real-analytic and strictly monotone on such a domain? For $\alpha\in \mathbb{R}$ is there an Borel probability measure $\mu$ on $\Lambda$ such that the Hausdorff dimension of $\mu$ is $b(\alpha)$ and $\int \phi d\mu=\alpha$, and, if such a measure exists, is it unique? In this paper, we provide positive answers to these questions.

Let $I:=[0,1]$. 
In this paper, for $A \subset I$, $\operatorname{Int}(A)$ and $\overline{A}$ denote its interior and closure in the Euclidean metric on $\mathbb{R}$.
A map $f:I\rightarrow I$ is said to be non-uniformly expanding interval map if $f$ satisfies the following conditions:
\begin{itemize}
    \item[(NEI1)] There exist a finite index set $\edge\subset\mathbb
    N$ with $\#\edge\geq 2$ and a family $\{\Delta_{i}\}_{i\in \edge}$ of  subintervals of $I$ such that for each $i,j\in\edge$ with $i\neq j$ we have $\text{int}(\Delta_i)\cap\text{int}(\Delta_j)=\emptyset$. 
    \item[(NEI2)] For all $i\in \edge$ the map $f|_{\Delta_i}:\Delta_i\rightarrow f(\Delta_i)$ is a $C^{1+\varepsilon}$ diffeomorphism for some $\varepsilon>0$ and $(0,1)\subset f(\Delta_i)\subset [0,1]$. Furthermore, 
    there exists a open set $W_i$ such that $\overline{\Delta_i}\subset W_i$ and $f|_{\Delta_i}$ extends to a $C^{1+\varepsilon}$ diffeomorphism $f_i$ from $W_i$ onto its images.
    \item[(NEI3)] There exists a non-empty set $\p\subset \edge$ of parabolic indexes such that for each $i\in \p$, the map 
    $f_i$ has a unique fixed point $x_i\in \overline{\Delta_i}$ satisfying $|f_i'(x_i)|=1$, 
    $|f_i'(x)|>1$ for all $x\in W_i\setminus\{x_i\}$
    and $f_i$ is $C^2$ on $W_i\setminus\{x_i\}$. Moreover, for all $i\in\h:=\edge\setminus\p$ the map $f_i$ is $C^2$ on $W_i$ and there exists $c>1$ such that for all $i\in\h$ and $x\in W_i$ we have $|f'_i(x)|>c$.
\end{itemize}
Let $f$ be a non-uniformly expanding interval map.
For each $i\in \edge$ we denote by $T_i$  the inverse of $f_i$. 
For each $n\in\mathbb{N}$ and $\omega\in \edge^n$ we set
\begin{align}\label{eq def basic set}
T_\omega :=T_{\omega_1}\circ\cdots\circ T_{\omega_n}
\text{ and }
\bar\Delta_\omega:=T_\omega([0,1]).
\end{align}
Then, by \cite[Proposition 3.1]{MixedJT} the Euclidean diameter  $|\bar \Delta_{\omega}|$ of the set $\bar\Delta_\omega$  converges to $0$, uniformly in all sequences, that is, 
\begin{align}\label{eq uniformly decay of sylinders}
\lim_{n\to\infty}\max_{\omega\in \edge^n}|\bar\Delta_\omega|=0.    
\end{align}
Therefore, since for each $n\in\mathbb{N}$ and $\omega\in \edge^n$ the set $\bar\Delta_\omega$ is compact, for each $\omega\in \edge^{\mathbb{N}}$ the set $\bigcap_{n\in\mathbb{N}}\bar\Delta_{\omega_1\cdots\omega_n}$ is singleton.
We define the limit set $\Lambda$ of $f$ by 
\[
\Lambda:=\bigcup_{\omega\in\edge^{\mathbb{N}}}\bigcap_{n\in\mathbb{N}}\bar\Delta_{\omega_1\cdots\omega_n}.
\]
In this paper, for $J\subset [0,1]$ we always assume that $J$ is endowed with the relative topology from $[0,1]$.

Next, we explain conditions regarding a induced map of a non-uniformly expanding interval map $f$.
For all $n\in\mathbb{N}$ and $\omega\in \edge^n$ we set 
\[
I_\omega:=\bar\Delta_\omega\cap\Lambda.
\]
Define
\[
\inducingdomain:=\bigcup_{i\in \h}I_i\cup
\bigcup_{i\in\p}\bigcup_{j\in \edge\setminus\{i\}}I_{ij}.
\]
We define the return time function $\return:\inducingdomain\rightarrow \mathbb{N}$ by
\[
\return(x):=\inf \{n\in\mathbb{N}:f^n(x)\in \inducingdomain\}
\]
and the induced map $\indumap:\{\return<\infty\}\rightarrow I$ by
\[
 \indumap(x):=f^{\return(x)}(x).
\]
This construction of the induced map $\indumap$ from $f$ is commonly used in the study of the Bowen-Series map associated with Kleinian Schottky groups with parabolic generators (see \cite{JKScuspidal}, \cite{Frechet}, \cite{jaerisch2022multifractal}, \cite{kessebohmer2004multifractaltobeappdated}).
The following conditions allow us to analyze $f$ by using $\indumap$:
\begin{itemize}
    \item[(F1)] There exist a constant $C\geq 1$ and a exponent $\exponent>0$ such that for all $n\in\mathbb{N}$ and $x\in \{\return=n\}$ we have 
    \[
    \frac{1}{C}\leq\frac{|\indumap'(x)|}{n^{1+\gamma(f)}}\leq C 
    \]
    \item[(F2)] The induced map $\indumap$ satisfies the Renyi condition, that is,
    \[
    \sup_{n\in\mathbb{N}}
    \sup_{j\in\p}\sup_{i\in \edge\setminus\{j\}}
    \sup_{x\in\{\return=n\}\cap I_{ij}}
    \frac{|\indumap''(x)|}{|\indumap'(x)|^2}<\infty.
    \]
\end{itemize}
    The induced map $\indumap$ is said to be admissible if $\indumap$ satisfies (F1) and (F2). 

    As in \cite{JJTOPnonuniformly}, our main examples are the following:
\begin{example}
    Let $f:I\rightarrow I$ be the Manneville-Pomeau map defined by 
    $f(x)=x+x^{1+\beta}$ mod $1$, where $\beta>0$. 
    The map $f$ has the parabolic fixed point at $0$, while it is expanding everywhere else. It is straightforward to verify that $f$ is a non-uniformly expanding interval map. 
    Furthermore, 
    by \cite{Thaler} (see also \cite{Nakaishi}), one can show that $\indumap$ satisfies (F1) with $\exponent=1/\beta$ and (F2).
\end{example}
\begin{example}
    Let $f:[0,1]\rightarrow[0,1]$ be defined by $f(x):=x/(1-x)$ for $0\leq x\leq 1/2$ and $f(x):=(2x-1)/x$ for $1/2< x\leq 1$. Then, the map $f$ has parabolic fixed points at $0$ and $1$ but is expanding everywhere else. It is not difficult to verify that $f$ is a non-uniformly expanding interval map. Moreover,
    $\indumap$ satisfies (F1) with $\exponent=1$ and (F2).
\end{example}

Let $\phi:\Lambda\rightarrow\mathbb{R}$ be a continuous function.
We define the induced potential $\indupote:\Lambda\cap\{\return<\infty\}\rightarrow\mathbb{R}$ of $\phi$ by 
\begin{align}\label{eq def indu pot}
\indupote(x):=\sum_{i=0}^{\return(x)-1}\phi(f(x)).
\end{align}
In this paper, we always require the following condition (see Definition \ref{def locally holder}):
\begin{itemize}
    \item[(H)] There exists $\beta>0$ such that $\indupote$ is locally H\"older with exponent $\beta$. 
\end{itemize}
We denote by $\mathcal{R}$ the set of all continuous function $\phi$ on $\Lambda$ satisfying (H).
The level set we consider is given by 
\[
\Lambda_\alpha:=\left\{x\in \Lambda:\lim_{n\to\infty}\frac{1}{n}\sum_{i=0}^{n-1}\phi(f^i(x))=\alpha\right\}\ \ (\alpha\in\mathbb{R}).
\]
We define the Birkhoff spectrum $b:\mathbb{R}\rightarrow [0,1]$ by 
\[
b(\alpha)=\dim_H(\Lambda_\alpha),
\]
where $\dim_H(\cdot)$ denotes the Hausdorff dimension with respect to the Euclidean metric on $\mathbb{R}$. We will use the following notations: For $i\in\p$ we set
\begin{align*}
&\alpha_i:=\phi(x_i),
\ \alpha_{\min}:=\inf_{\mu\in M(f)}\left\{\int\phi d\mu\right\}\text{ and }
    \alpha_{\max}:=\sup_{\mu\in M(f)}\left\{\int\phi d\mu\right\},
\end{align*}
where $M(f)$ denotes the set of $f$-invariant Borel probability measures on $\Lambda$. 
Note that $f$ has full-branches. Hence, by the convexity of the set $M(f)$ and Birkhoff's ergodic theorem, we have 
$
\{\alpha\in \mathbb{R}:B(\alpha)\neq \emptyset\}=[\alpha_{\text{min}},\alpha_{\text{max}}].
$

For $\mu\in M(f)$ we define $\lambda(\mu):=\int \log |f'|d\mu$ and denote by $h(\mu)$ the measure-theoretic entropy defined as \cite{walters2000introduction}.
Let $\delta:=\dim_H(\Lambda)$. If $\exponent>2/\delta-1$
then there exists a unique ergodic measure $\mu_\delta\in M(f)$ of full dimension, that is, $\delta=\inf \{\dim_H(Z):\mu_\delta(Z)=1\}$. 
We define 
\begin{align}\label{eq def frat}
\frat:=\left\{
\begin{array}{cc}
[\underline{\alpha}_\p,\overline{\alpha}_\p] 
&\text{ if }\exponent\leq 2/\delta-1\\
\left[\min\{\int \phi\mu_\delta,\underline{\alpha}_\p\},
\max\{\int \phi\mu_\delta,\overline{\alpha}_\p\}\right] 
& \text{ if }\exponent> 2/\delta-1
,
\end{array}
\right.
\end{align}
where $\underline{\alpha}_\p:=\min_{i\in\p}\{\alpha_i\}$ and $\overline{\alpha}_\p:=\max_{i\in\p}\{\alpha_i\}$.

\begin{thm}\label{thm main}
Let $f$ is a non-uniformly expanding interval map having the admissible induced map $\indumap$ and $\phi\in \mathcal{R}$.
Then, we have the following 
\begin{itemize}
    \item[(1)] For all $\alpha\in (\alpha_{\min} , \alpha_{\max} ) \setminus \frat$ there exists a unique ergodic measure $\mu_\alpha\in M(f)$ such that $\lambda(\mu_\alpha)>0$ and 
$
b(\alpha)={h(\mu_\alpha)}/\lambda(\mu_\alpha)$.
\item[(2)] The function $b$ is real-analytic on $(\alpha_{\min} , \alpha_{\max} ) \setminus \frat$
\item[(3)] For all $\alpha\in \frat$ we have $b(\alpha)=\delta$ and for all $\alpha\in (\alpha_{\min} , \alpha_{\max} ) \setminus \frat$ we have $b(\alpha)<\delta$. Moreover, $b$ is strictly increasing on $(\alpha_{\min},\min \frat)$ and strictly decreasing on $(\max \frat,\alpha_{\max})$.
\end{itemize}
\end{thm}

Our main difficulty arises from the lack of uniform hyperbolicity caused by the presence of parabolic fixed points. In the non-uniformly hyperbolic setting, the thermodynamic formalism does not work well as in  the uniformly hyperbolic case.
Namely, in general, for $\phi\in \mathcal{R}$ and $(b,q)\in\mathbb{R}^2$ an equilibrium measure $\mu$ for the potential $-q\phi-b\log |f'|$ with $\lambda(\mu)>0$  does not exist and the pressure function $(b,q)\mapsto P(-q\phi-b\log |f'|)$ is not differentiable on $\mathbb{R}^2$, where $P(-q\phi-b\log |f'|)$ denotes the topological pressure of $-q\phi-b\log |f'|$ (see \eqref{eq def non-induced pressure}). 
This difficulty makes it challenging to directly apply the thermodynamic approach developed by Barreira and Saussol \cite{BarreiraSaussol} for the uniformly hyperbolic case.
To overcome these difficulties we first determine the domain $\mathcal{N}\subset \mathbb{R}^2$ such that for all $(b,q)\in \mathcal{N}$ an equilibrium measure $\mu$ for the potential $-q\phi-b\log |f'|$ with $\lambda(\mu)>0$ exists uniquely, and the  pressure function is real-analytic on $\mathcal{N}$. 
We then prove our main results by combining the above results on the thermodynamic formalism with careful arguments involving the variational principle and the convexity of the pressure function.     
For non-uniformly expanding maps with a single parabolic fixed point and a continuous potential, Climenhaga \cite[Theorem 3.9]{simultaneous} obtained a result relating the topological pressure to the Birkhoff spectrum.   
However, in the non-uniformly hyperbolic setting, the thermodynamic formalism does not work effectively, as explained above. Consequently, unlike in the uniformly hyperbolic case, it is difficult to derive our results directly from \cite[Theorem 3.9]{simultaneous}. 
Moreover, while the assumption that non-uniformly expanding maps have a single parabolic fixed point plays an important role in his proofs (see \cite[Theorem 3.3]{simultaneous}), our main results allow for several parabolic fixed points.
We also note that Iommi and Jordan \cite{IommiJordanquotients} proved (2) of Theorem \ref{thm main} for Manneville-Pomeau type maps with a single parabolic fixed point and $\phi\in \mathcal{R}$ by relating the Birkhoff spectrum to the multifractal spectrum defined via the limit of the quotient of the Birkhoff sums of an induced potential of $\phi$ and a return time function with respect to an inducing domain.
However, we emphasize that our proof differs substantially from \cite{IommiJordanquotients}.

The structure of the paper is as follows. In Section \ref{sec preliminary}, we introduce the tools that will be used in Sections \ref{sec thermodynamic} and \ref{sec multifractal}. Section \ref{sec thermodynamic} is devoted to developing the thermodynamic formalism for a non-uniformly expanding interval map. In Section \ref{sec multifractal}, we prove Theorem \ref{thm main}.

\textbf{Notations.} Throughout we shall use the following notation:
For a index set $\mathcal{Q}$ and $\{a_{q}\}_{q\in\mathcal{Q}},\{b_{q}\}_{q\in\mathcal{Q}}\subset[0,\infty]$ we write $a_q\ll b_q$ if there exists a constant $C\geq 1$ such that for all $q\in \mathcal{Q}$ we have $a_q\leq Cb_q$. If we have $a_q\ll b_q$ and $b_q\ll a_q$ then we write $a_q\asymp b_q$.
For a probability space $(X,\mathcal{B})$, a probability measure $\mu$ on $(X,\mathcal{B})$ and a measurable function $\psi:X\rightarrow \mathbb{R}$ we set
$\mu(\psi):=\int \psi d\mu$.

\section{Preliminary}\label{sec preliminary}
Let $f$ be a non-uniformly expanding interval map having the admissible induced map $\indumap$ and let $\phi:\Lambda\rightarrow \mathbb R$ be a continuous function. We first introduce the topological pressure of $\phi$ with respect to the dynamical system $f:\Lambda\rightarrow \Lambda$.
The topological pressure of $\phi$ is defined by 
\begin{align}\label{eq def non-induced pressure}
P(\phi):=\sup\left\{h(\mu)+\mu(\phi):\mu\in M(f)\right\}.    
\end{align}
We say that $\mu\in M(f)$ is an equilibrium measure for $\phi$ if $\mu$ satisfies
$
P(\phi)=h(\mu)+\mu(\phi).
$
We are interested in the pressure function $p:\mathbb{R}^2\rightarrow \mathbb{R}$ given by
\[
p(b,q):=P(-q\phi-b\log|f'|).
\]
Let 
\[
\mathcal{N}:=\left\{(b,q)\in \mathbb{R}^2:p(b,q)>\max_{i\in\p}\left\{-q\alpha_i\right\}\right\}.
\]
Note that since the function $p$ is convex on $\mathbb{R}^2$, the set $\mathcal{N}$ is open.

Next, we describe a coding space of the induced map $\indumap$.
For $n\in\mathbb{N}$ we define 
\begin{align*}
\edgepart_n:=\{\omega\in \edge^{n+2}:\text{Int}(\Delta_\omega)\cap\Lambda\subset \{\return=n\}\}
\text{ and }
   \edgesindu:=\bigcup_{n\in\mathbb{N}}\edgepart_n.
\end{align*}
We first notice that $\{\return<\infty\}=\bigcup_{\omega\in \edgesindu}I_\omega$.
Moreover, $\omega\in \edgepart_1$ has the form
\begin{align}\label{eq form edges 1}
&\omega=\omega_1\omega_2\omega_3 \text{ for some }\omega_1\in\edge,
\omega_2\in\left\{
\begin{array}{cc}
\edge\setminus\{\omega_1\}
&\text{ if }\omega_1\in \p\\
\edge
& \text{ otherwise }
\end{array}
\right.
\\
&
\nonumber 
\text{ and }
\omega_2\in\left\{
\begin{array}{cc}
\edge\setminus\{\omega_2\}
&\text{ if }\omega_2\in \p\\
\edge
& \text{ otherwise },
\end{array}
\right.
\end{align}
and $\omega\in \edgepart_n$ ($n\geq 2$) has the form
\begin{align}\label{eq foem edges 2}
    \omega=\omega_1i^{n}\omega_{n+2} \text{ for some } i\in\p\ \text{and }\omega_1,\omega_{n+2}\in\edge\setminus\{i\}.
\end{align}
Also, notice that for all $n\in\mathbb{N}$ and $\omega\in \edgepart_n$ we have
\begin{align}\label{eq markov}
    \indumap(I_\omega)=f^{n}(I_\omega)=I_{\omega_{n+1}\omega_{n+2}}.
\end{align}
Therefore, 
if $\indumap(I_\omega)\cap I_{\omega'}$ ($\omega, \omega'\in\edgesindu$) has non-empty interior then $I_{\omega'}\subset \indumap(I_\omega)$.
This implies that
$\indumap:\{\return<\infty\}\rightarrow \inducingdomain$ is a Markov map with the countable Markov partition $\{I_\omega\}_{\omega\in \edgesindu}$.
We define the incidence matrix $B:\edgesindu\times \edgesindu\rightarrow\{0,1\}$
by $B_{\omega,\omega'}=1$ if $I_{\omega'}\subset \indumap(I_{\omega})$ and
$B_{\omega,\omega'}=0$ otherwise. Define the coding space 
\[
\inducoding:=\lbrace\omega\in \edgesindu^{\mathbb{N}}:B_{\omega_{n},\omega_{n+1}}=1,\ n\in\mathbb{N}\rbrace.
\]
We denote by $\inducoding^{n}$ $(n\in\mathbb{N})$ the set of all admissible
words of length $n$ with respect to the incidence matrix $B$ and
by $\inducoding^{*}$ the set of all admissible words which have a finite length
(i.e. $\inducoding^*=\cup_{n\in\mathbb{N}}\inducoding^{n}$). 
For $\omega\in \inducoding^{n}$ ($n\in\mathbb{N}$) we define the cylinder set of $\omega$ by
$[\omega]:=\{\tau\in\inducoding:\tau_{i}=\omega_{i},1\leq i\leq n\}$.
We endow $\inducoding$
with the metric $d$ defined by $d(\omega,\omega')={e^{-k}}$ $\text{if}\ \omega_{i}=\omega_{i}'\ \text{for\ all}\ i=1,\cdots,k-1\ \text{and}\ \omega_{k}\neq\omega_{k}'$
and $d(\omega,\omega')=0$ otherwise.   Let $\shift:\inducoding\rightarrow\inducoding$
be the left-shift map. 
By \eqref{eq uniformly decay of sylinders}, 
for each $\omega\in \inducoding$ the set 
$\bigcap_{n\in\mathbb{N}}I_{\omega_1\cdots\omega_n}$
is a singleton. Thus, we can define the coding map $\codingmap:\inducoding\rightarrow \codingmap(\inducoding)$
by 
\[
\{\codingmap(\omega)\}=\bigcap_{n\in\mathbb{N}}I_{\omega_1\cdots\omega_n} \text{ and set }\indulimit=\codingmap(\inducoding).
\]
Then, we have $\indumap(\indulimit)=\indulimit$. 
We denote by $M(\indumap)$ the set of $\indumap$-invariant Borel probability measures on $\indulimit$ and by $M(\shift)$ the set of $\shift$-invariant Borel probability measures on $\inducoding$.
\begin{rem}\label{rem measurable bijection}
    We notice that $\codingmap$ is continuous and one-to-one except on the preimage of the countable set $J_0:=\bigcup_{n=0}^\infty \indumap^{-n}(\bigcup_{\omega\in \edgesindu}\partial \bar\Delta_{\omega})$, where it is at most two-to-one. Furthermore, we have $\indumap\circ\codingmap=\codingmap\circ\shift$ on $\inducoding\setminus\codingmap^{-1}(J_0)$ and the restriction of $\codingmap$ to $\inducoding\setminus\codingmap^{-1}(J_0)$ has a continuous inverse.  Thus, $\codingmap$ induces a measurable bijection between $\inducoding\setminus\codingmap^{-1}(J_0)$ and $\indulimit\setminus J_0$. Furthermore, by applying the argument in the proof of \cite[Lemma 3.5]{jaerisch2022multifractal}, for any $\measure\in M(\indumap)$ there exists $\codingmeasure\in M(\shift)$ such that $\measure=\codingmeasure\circ\codingmap^{-1}$ and $h(\measure)=h(\codingmeasure)$, and for $\codingmeasure\in M(\shift)$ we have $\codingmeasure\circ\codingmap^{-1}\in M(\indumap)$ and  $h(\codingmeasure\circ\codingmap^{-1})=h(\codingmeasure)$. 
\end{rem}

We will recall results from the thermodynamic formalism for ($\inducoding,\shift$).
For a function $\tilde\psi:\inducoding\rightarrow \mathbb{R}$ and $n\in\mathbb{N}$ we set 
$S_n(\tilde\psi):=\sum_{k=0}^{n-1} \tilde\psi\circ\shift^k$.
 For a continuous function $\tilde\psi$ on $\inducoding$ the topological pressure of $\tilde\psi$ introduced by Mauldin and Urba\'nski \cite{mauldin2003graph} is given as
\[
\tilde{P}(\tilde\psi):=\lim_{n\to\infty}\frac{1}{n}\log\sum_{\omega\in \inducoding^n}\exp\left(
\sup_{\tau\in[\omega]} 
S_n(\tilde\psi)(\tau)
\right).
\]
Let $\phi:\Lambda\rightarrow\mathbb{R}$ be a continuous function and let $\indupot$ be the induced potential defined by \eqref{eq def indu pot}. For some $\beta>0$ the induced potential $\indupot$ is said to be locally H\"older with exponent $\beta$ if  
\begin{align}\label{def locally holder}
    \sup_{n\in\mathbb{N}}\sup_{\omega\in \inducoding^n}\sup\{|\indupot\circ\codingmap(\tau)-\indupot\circ\codingmap(\tau')|(d(\tau,\tau'))^{-\beta}:\tau,\tau'\in [\omega]\}<\infty.
\end{align}
Since $\indumap$ is uniformly expanding, that is, there exists $c>1$ such that for all $x\in \indulimit$ we have $|\indumap'(x)|>c$, and $\indumap$ satisfies the Renyi condition
(F2), there exists $\beta>0$ such that $\log|\indumap'|$ is locally H\"older with exponent $\beta$. In the following, for $\phi\in \mathcal{R}$ we assume that, if necessary by replacing the exponent $\beta$ with a smaller number, $\log|\indumap'|$ and $\indupote$ are locally H\"older with $\beta$. Also, note that, since for each $\omega\in \edgesindu$, $\return\circ\codingmap$ is constant on $[\omega]$, the return time function $\return$ is  locally H\"older with exponent $\beta$. 
By \eqref{eq form edges 1}, \eqref{eq foem edges 2} and \eqref{eq markov}, it is not difficult to verify that $\inducoding$ is finitely primitive, that is, there exist $n\in\mathbb{N}$ and a finite set $\Omega\subset \inducoding^n$ such that for all $\tau,\tau'\in \edgesindu $ there is $\omega=\omega(\tau,\tau')\in \Omega$ for which $\tau\omega\tau'\in \inducoding^*$. These conditions enable us to apply results from thermodynamic formalism for the countable Markov shift ($\inducoding,\shift$) as  presented in \cite[Section 2]{mauldin2003graph} (see also \cite[Section 17, 18 and 20]{Urbanskinoninvertible}).

Let $\phi\in \mathcal{R}$. Define the pressure function $p:\mathbb{R}^3\rightarrow \mathbb{R}$ by 
\[
\indupressure(b,q,s):=\tilde{P}(-q\indupote\circ\codingmap-b\log|\indumap'\circ\codingmap|-s\return\circ\codingmap).
\]
\begin{thm}{\cite[Theorem 2.1.8]{mauldin2003graph}}\label{thm variational principle induce} For all $(b,q,s)\in \mathbb{R}^3$ we have
\[
\indupressure(b,q,s)=\sup_{\codingmeasure}\left\{h(\codingmeasure)+
\codingmeasure\left(\left(-q\indupote-b\log|\indumap'|-s\return\right)\circ\codingmap\right) \right\},
\]
where the supremum is taken over the set of measures $\codingmeasure\in M(\shift)$ satisfying $\codingmeasure((-q\indupote-b\log|\indumap'|-s\return)\circ\codingmap)>-\infty$. 
\end{thm}
Let 
\[
Fin:=\{(b,q,s)\in \mathbb{R}^3:\indupressure(b,q,s)<\infty\}.
\]
\begin{prop}{\cite[Proposition 2.1.9]{mauldin2003graph}}\label{prop finite pressure}
 $(b,q,s)\in Fin$ if and only if 
 \[
 \sum_{\omega\in \edgesindu}\exp\left(\sup_{\tau\in [\omega]}\left\{\left(-q\indupote-b\log|\indumap'|-s\return\right)(\codingmap(\tau))\right\}\right)<\infty.
 \]
\end{prop}
For $(b,q,s)\in Fin$ a measure $\codingmeasure\in M(\shift)$  is called a Gibbs measure for $-q\indupote\circ\codingmap-b\log|\indumap'\circ\codingmap|-s\return\circ\codingmap$ if there exists a constant $Q\geq 1$ such that for every $n\in\mathbb{N}$, $\omega\in \inducoding^n$ and $\tau\in[\omega]$ we have
\begin{align}\label{eq gibbs}
\frac{1}{Q}\leq \frac{\codingmeasure([\omega])}{\exp(S_n((-q\indupote-b\log|\indumap'|-s\return)\circ\codingmap)(\tau)-\indupressure(b,q,s)n)}\leq Q    
\end{align}
\begin{thm}\cite[Theorem 2.2.4 and Corollary 2.7.5]{mauldin2003graph}\label{thm Gibbs induce}
    For $(b,q,s)\in Fin$ there exists a unique Gibbs measure $\codingmeasure_{b,q,s}\in M(\shift)$ for $-q\indupote\circ\codingmap-b\log|\indumap'\circ\codingmap|-s\return\circ\codingmap$. Moreover, $\codingmeasure_{b,q,s}$ is ergodic.
\end{thm}
In the following, for $(b,q,s)\in Fin$ we denote by $\codingmeasure_{b,q,s}$ the unique $\shift$-invariant Gibbs measure obtained in Theorem \ref{thm Gibbs induce}.
For $(b,q,s)\in \mathbb{R}^3$ we say that $\codingmeasure\in M(\shift)$ is an equilibrium measure for $-q\indupote\circ\codingmap-b\log|\indumap'\circ\codingmap|-s\return\circ\codingmap$ if we have $\codingmeasure((-q\indupote-b\log|\indumap'|-s\return)\circ\codingmap)>-\infty$ and $\indupressure(b,q,s)=h(\codingmeasure)+\codingmeasure((-q\indupote-b\log|\indumap'|-s\return)\circ\codingmap)$.
\begin{thm}{\cite[Theorem 2.2.9]{mauldin2003graph}}\label{thm equilibrium state induce}
  Let  $(b,q,s)\in Fin$. If $\codingmeasure_{b,q,s}((-q\indupote-b\log|\indumap'|-s\return)\circ\codingmap)>-\infty$ then $\codingmeasure_{b,q,s}$ is the unique equilibrium measure for $-q\indupote\circ\codingmap-b\log|\indumap'\circ\codingmap|-s\return\circ\codingmap$.
\end{thm}
For $\codingmeasure\in M(\shift)$ we define $\lambda(\codingmeasure)=\codingmeasure( \log|\indumap'\circ\codingmap|)$
\begin{thm}{\cite[Theorem 2.6.12 and Proposition 2.6.13]{mauldin2003graph} (see also \cite[Theorem 20.1.12]{Urbanskinoninvertible})} \label{thm regularity of induced pressure}The function $(b,q,s)\mapsto\indupressure(b,q,s)$ is real-analytic on $\text{Int}(Fin)$. Moreover, we have Ruell's formula, that is, for $(b,q,s)\in \text{Int}(Fin)$,
$
\frac{\partial}{\partial b}\indupressure(b,q,s)=-\lambda(\codingmeasure_{b,q,s}),
$
$
\frac{\partial}{\partial q}\indupressure(b,q,s)=-\codingmeasure(\indupote\circ\codingmap),
$
$
\frac{\partial}{\partial s}\indupressure(b,q,s)=-\codingmeasure(\return\circ\codingmap).
$
\end{thm}
For $\codingmeasure\in M(\shift)$ we define $\measure:=\codingmeasure\circ\codingmap^{-1}$. Then, by Remark \ref{rem measurable bijection}, we have $\measure\in M(\indumap)$. For $\tilde \nu\in M(\indumap)$ with $\tilde\nu(\return)<\infty$ we define
\begin{align}\label{eq def lift}
\nu:=\frac{1}{\tilde\nu( \return )}\sum_{n=0}^{\infty}\sum_{k=n+1}^\infty\tilde\nu|_{\{\return=k\}}\circ f^{-n}.    
\end{align}
Since $\indumap$ is a first return map of $f$, it is well-known that for $\tilde\nu\in M(\indumap)$ with $\tilde\nu( \return )<\infty$ we have $\nu\in M(f)$ (for example see \cite[Proposition 1.4.3]{viana}). Also, if $\tilde\nu\in M(\indumap)$ with $\tilde\nu( \return )<\infty$ is ergodic then we have Abramov-Kac's formula  (see \cite[Theorem 2.3]{PesinSenti}):
\begin{align}\label{eq Abramov-Kac's formula}
    \tilde\nu( \return ) h(\nu)=h(\tilde \nu) \text{ and }\tilde\nu( \return )\nu(\psi)=\tilde\nu(\tilde\psi)
\end{align}
for a continuous function $\psi$ on $\Lambda$.  Conversely, for a ergodic measure $\nu\in M(f)$ with $\nu(\indulimit)>0$ and $\tilde \nu:=\nu|_{\indulimit}/\nu(\indulimit)$  we have 
\begin{align}\label{eq classical Abramov-Kac's formula}
    \tilde\nu( \return ) h(\nu)=h(\tilde \nu) \text{ and }\tilde\nu( \return )\nu(\psi)=\tilde\nu( \tilde\psi)
\end{align}
for a continuous function $\psi$ on $\Lambda$. 
Define, for a finite set $\mathcal{L}$ and $\{\nu_{\ell}\}_{\ell\in \mathcal{L}}\subset M(f)$,
$
\text{Conv}(\{\nu_{\ell}\}_{\ell\in \mathcal{L}}):
=
\{\sum_{\ell\in\mathcal{L}}p_\ell\nu_{\ell}:\{p_\ell\}_{\ell\in\mathcal{L}}\subset[0,1],\ \sum_{\ell\in\mathcal{L}}p_\ell=1
\}.$
\begin{lemma}\label{lemma equivalent condition not liftable}
    Let $\nu\in M(f)$. Then $\nu(\indulimit)=0$ if and only if $\nu\in \conv$, where $\delta_{x_i}$ ($i\in\p$) denotes the Dirac measure at $x_i$. 
\end{lemma}
The proof of Lemma \ref{lemma equivalent condition not liftable} is straightforward and is therefore omitted.

\section{thermodynamic formalism}\label{sec thermodynamic}
We denote by $\maps$ the set of non-uniformly expanding interval maps having the admissible induced map $\indumap$.  For a function $\tilde\psi$ on $\inducoding$ and $\omega\in \inducoding^*$ we set $\tilde \psi([\omega]):=\sup_{\tau\in\omega}\tilde\psi(\tau)$. 
In this section, we assume throughout that $f\in \maps$ and $\phi\in \mathcal{R}$.
Recall that $\alpha_i:=\phi(x_i)$ ($i\in\p$).
For $q\in \mathbb{R}$ we set 
\[
\LB:= \max_{i\in\p}\{-q\alpha_i\}.
\]

\begin{lemma}\label{lem finite pressure on a special parameter}
    For all $(b,q)\in\mathbb{R}^2$ and $s\in(\LB,\infty)$ we have $(b,q,s)\in Fin$. 
   In particular, for all $(b,q)\in\mathcal{N}$ we have $(b,q,p(b,q))\in Fin$. 
\end{lemma}

\begin{proof}
Fix $(b,q)\in\mathbb{R}^2$ and $s\in( \LB,\infty)$.
We take a small $\epsilon>0$ with $ \LB+\epsilon<s$. 
Since $\phi$ is continuous on $\Lambda$, there exists $N\geq 2$ such that for all $i\in\p$, $n\geq N$ and $x\in I_{i^n}$ we have $|q\phi(x)-q\alpha_i|<\epsilon$. Thus, by (F1), we obtain
\begin{align*}
&\sum_{n=N}^{\infty}\sum_{\omega\in\edgepart_n}e^{\left(-q\indupote-b\log|\indumap'|- s \return\right)\circ\codingmap([\omega])}
\asymp
\sum_{n=N}^{\infty}
\sum_{\omega\in\edgepart_n}n^{-b(1+\exponent)}
e^{(-q\phi- s )\circ\codingmap([\omega])}
\\&\exp
\left(
\left(
\sum_{k=1}^{n-N+1}(-q\phi- s )\circ\codingmap\circ\shift^k
\right)
([\omega])
\right)
e^{\left(\sum_{k=n-N+2}^{n-1}(-q\phi- s )\circ\codingmap\circ\shift^k\right)([\omega])}
\\&
\asymp
\sum_{n=N}^{\infty}\sum_{\omega\in\edgepart_n}\frac{
e^{\left(
\sum_{k=1}^{n-N+1}(-q\phi- s )\circ\codingmap\circ\shift^k
\right)([\omega])}}{n^{b(1+\exponent)}}
\ll 
\sum_{n=N}^{\infty}\frac{
e^{( \LB+\epsilon- s )n
}}{n^{b(1+\exponent)}}<\infty.
\end{align*}
Therefore, by Proposition \ref{prop finite pressure}, we are done.
\end{proof}

For $(b,q)\in \mathcal{N}$ we write $\codingmeasure_{b,q}:=\codingmeasure_{b,q,\indupressure(b,q)}$.

\begin{lemma}\label{lemma finite return time}
    For all $(b,q)\in\mathcal{N}$ and $\ell\in\mathbb{N}$ we have $\codingmeasure_{b,q}(| \return\circ\codingmap|^\ell)<\infty$. Moreover, we have 
    $\codingmeasure_{b,q}( |\indupote\circ\codingmap|^{\ell})<\infty$ and $\codingmeasure_{b,q}(|\log|\indumap'||^\ell)<\infty$. In particular, $\codingmeasure_{b,q}$ is the unique equilibrium measure for $-q\indupote\circ\codingmap-b\log|\indumap'\circ\codingmap|-p(b,q)\return\circ\codingmap$.
\end{lemma}
\begin{proof}
    Let $(b,q)\in\mathcal{N}$ and let $\ell\in\mathbb{N}$.  
We take a small $\epsilon>0$ with $ \LB+2\epsilon<p(b,q)$. 
Since $\lim_{x\to\infty}x^\ell e^{-\epsilon x}=0$, there exists $N\geq 2$ such that for all $n\geq N$ we have $n^\ell\leq e^{\epsilon n}$. Therefore, by \eqref{eq gibbs}, we obtain
\begin{align*}
&   \sum_{n=N}^\infty \sum_{\omega\in\edgepart_n}\int_{[\omega]} |\return\circ\codingmap |^\ell d\codingmeasure_{b,q}
   = \sum_{n=N}^\infty \sum_{\omega\in\edgepart_n}n^\ell \codingmeasure_{b,q}([\omega])
  \\& \asymp\sum_{n=N}^\infty \sum_{\omega\in\edgepart_n}n^\ell
   \exp\left((-q\indupote-b\log|\indumap'|-p(b,q)\return)\circ\codingmap([\omega])-\indupressure(b,q,p(b,q))\right)
   \\&\leq e^{-\indupressure(b,q,p(b,q))} \sum_{n=N}^\infty \sum_{\omega\in\edgepart_n}
   \exp\left((-q\indupote-b\log|\indumap'|-(p(b,q)-\epsilon)\return)\circ\codingmap([\omega])\right)
\end{align*}
By the same argument in the proof of Lemma \ref{lem finite pressure on a special parameter}, we can show that\\ $\sum_{n=N}^\infty \sum_{\omega\in\edgepart_n}
   e^{(-q\indupote-b\log|\indumap'|-(p(b,q)-\epsilon)\return)\circ\codingmap([\omega])}<\infty$, which yields $\codingmeasure_{b,q}(|\return\circ\codingmap|^\ell )<\infty$. 
   Since $\edge$ is finite, $\Lambda$ is compact and thus,   $
   \codingmeasure_{b,q}( |\indupote\circ\codingmap|^\ell)\leq \max_{x\in\Lambda}|\phi(x)|^\ell\codingmeasure_{b,q}(| \return\circ\codingmap|^\ell) <\infty.
   $
   Also,
   by a similar argument, we can show that $\codingmeasure_{b,q}(|\log|\indumap'||^\ell)<\infty$.
\end{proof}
For $(b,q)\in \mathcal{N}$ we define the measures $\measure_{b,q}:=\codingmeasure_{b,q}\circ\codingmap^{-1}$ on $\indulimit$ and 
$
\mu_{b,q}:=(\measure_{b,q}(\return ))^{-1}\sum_{n=0}^{\infty}\sum_{k=n+1}^\infty\measure_{b,q}|_{\{\return=k\}}\circ f^{-n}$ on $\Lambda$.    
Then, by Lemma \ref{lemma finite return time}, Remark \ref{rem measurable bijection} and \eqref{eq Abramov-Kac's formula}, we obtain
\begin{align}\label{eq induced pressure is less than zero}
&\indupressure(b,q,p(b,q))=h(\codingmeasure_{b,q})+\codingmeasure_{b,q}((-q\indupote-b\log|\indumap'|-p(b,q)\return)\circ\codingmap )
    \\&
    \nonumber
    =\measure_{b,q}(\return)
    \left(
    h(\mu_{b,q})+\mu_{b,q}(-q\phi-b\log|f'|)-p(b,q)
    \right) \leq 0.
\end{align}
Moreover, we obtain the following:
\begin{thm}\label{thm uniquness and existence of the equilibrium state}
     For $(b,q)\in \mathcal{N}$ we have that $\indupressure(b,q,p(b,q))=0$. Furthermore,  for $(b,q)\in \mathcal{N}$,  $\mu_{b,q}$ is the unique equilibrium measure for $-q\phi-b\log |f'|$.
\end{thm}
\begin{proof}
Let $(b,q)\in\mathcal{N}$. We first show that $\indupressure(b,q,p(b,q))=0$. By \eqref{eq induced pressure is less than zero}, it is enough to show that $\indupressure(b,q,p(b,q))\geq 0$. Let $\epsilon>0$ be a small number with $ \LB+\epsilon<p(b,q)$. By \cite[Corollary 9.10.1]{walters2000introduction}, there exists a ergodic measure $\mu\in M(f)$ such that we have 
\[
h(\mu)+\mu (-q\phi-b\log|f'|)>p(b,q)-\epsilon.
\]
Note that $\mu\notin \conv$. Indeed, if $\mu\in \conv$ then we have 
$h(\mu)+\mu (-q\phi-b\log|f'|)+\epsilon\leq  \LB+\epsilon<p(b,q)$ which yields a contradiction. Thus, by Remark \ref{rem measurable bijection}, Lemma \ref{lemma equivalent condition not liftable} and \eqref{eq classical Abramov-Kac's formula}, we obtain 
\begin{align}\label{eq proof of uniquness existence larger than zero}
    &\indupressure(b,q,p(b,q)-\epsilon)\geq h(\measure)+\measure(-q\indupote-b\log |\indumap'|-(p(b,q)-\epsilon)\return)
    \\&=\measure( \return )\left(h(\mu)+\mu (-q\phi-b\log |f'|)-(p(b,q)-\epsilon)\right)>0, \nonumber
\end{align}
where $\measure=\mu|_{\indulimit}/\mu(\indulimit)$. 
By Lemma \ref{lem finite pressure on a special parameter}, 
the function $s\mapsto\indupressure(b,q,s)$ is continuous on $( \LB,\infty)$. Hence, by \eqref{eq proof of uniquness existence larger than zero} and \eqref{eq induced pressure is less than zero}, we obtain $\indupressure(b,q,p(b,q))=0$. Moreover, by \eqref{eq induced pressure is less than zero}, the measure $\mu_{b,q}$ is an equilibrium measure for $-q\phi-b\log|f'|$.

Next, we shall show that $\mu_{b,q}$ is the unique equilibrium measure for $-q\phi-b\log|f'|$. Let $\nu$ be an equilibrium measure for $-q\phi-b\log|f'|$.
By the ergodic decomposition theorem (see \cite[Theorem 5.1.3]{viana}), we may assume that $\nu$ is ergodic.  As above, we have $\nu\notin \conv$ and thus, $\nu(\indulimit)>0$. 
Let $\tilde\nu=\nu|_{\indulimit}/\nu(\indulimit)$.
By Remark \ref{rem measurable bijection}, there exists $\tilde \nu'\in M(\shift)$ such that $\tilde \nu=\tilde \nu'\circ\codingmap^{-1}$ and $h(\tilde \nu)=h(\tilde \nu')$.
Then, $\tilde\nu'$ is an equilibrium measure for $(-q\indupote-b\log|\indumap'|-p(b,q)\return)\circ\codingmap$. Indeed, by Theorem \ref{thm variational principle induce}, Remark \ref{rem measurable bijection}, \eqref{eq classical Abramov-Kac's formula} and \eqref{eq induced pressure is less than zero}, we have 
\begin{align*}
   &0\geq\indupressure(b,q,p(b,q))
   \geq h(\tilde\nu')+\tilde\nu'((-q\indupote-b\log |\indumap'|-p(b,q)\return)\circ\codingmap)
    \\&=\tilde\nu( \return)\left(h(\nu)+\nu (-q\phi-b\log |f'|)-p(b,q)\right)=0. \nonumber    
\end{align*}
Therefore, by Theorem \ref{thm equilibrium state induce}, we obtain $\tilde \nu'=\codingmeasure_{b,q}$ and thus, $\nu=\mu_{b,q}$.
\end{proof}

For two function $\psi_1,\psi_2:\indulimit\rightarrow\mathbb{R}$ and $(b,q)\in \mathcal{N}$ we define the asymptotic variance of $\psi_1$ and $\psi_2$ by
\[
\sigma^2_{b,q}(\psi_1,\psi_2):=    \lim_{n\to\infty}\frac{1}{n}
    \codingmeasure_{b,q}\left( S_n\left(\psi_1\circ\codingmap-\codingmeasure_{b,q}(\psi_1\circ\codingmap) \right)
    S_n\left(\psi_2\circ\codingmap-\codingmeasure_{b,q}(\psi_2\circ\codingmap)\right)\right)
\]
when the limit exists. If $\psi_1=\psi_2$ then we write $\sigma^2_{b,q}(\psi_1):=\sigma^2_{b,q}(\psi_1,\psi_2)$.

\begin{thm}\label{thm regularity of non induced pressure}
The pressure function $(b,q)\mapsto p(b,q)$ is real-analytic on $\mathcal{N}$ and for $(b,q)\in \mathcal{N}$ we have
    \begin{align}\label{eq ruell's formula noninduced}
    \frac{\partial}{\partial b}p(b,q)=-\lambda(\mu_{b,q}) \text{ and } \frac{\partial}{\partial q}p(b,q)=-\mu_{b,q}( \phi ). 
    \end{align} 
    Moreover, for $(b,q)\in \mathcal{N}$ we have $\frac{\partial^2}{\partial q^2}p(b,q)=0$ if and only if $\alpha_{\min}=\alpha_{\max}$.  
\end{thm}
\begin{proof}
Let $(b_0,q_0)\in\mathcal{N}$. By Lemma \ref{lem finite pressure on a special parameter}, there exists a open neighborhood $O\subset \mathbb{R}^3$ of $(b_0,q_0,p(b_0,q_0))$ such that for all $(b,q,s)\in O$ we have $\indupressure(b,q,s)<\infty$. Also, by Theorem \ref{thm regularity of induced pressure}, we have 
$
\left.
\frac{\partial}{\partial s}\indupressure(b,q,s)\right|_{(b,q,s)=(b_0,q_0,p(b_0,q_0))}=-\codingmeasure_{b_0,q_0}( \return\circ\codingmap )<0.
$
Therefore, by the implicit function theorem and Theorem \ref{thm regularity of induced pressure}, the function $p$ is real-analytic at $(b_0,q_0)$. Moreover, the implicit function theorem, Theorem \ref{thm regularity of induced pressure} and \eqref{eq Abramov-Kac's formula} gives
\begin{align*}
    &\frac{\partial}{\partial b}p(b,q)=\frac{-\lambda(\codingmeasure_{b,q})}{\codingmeasure_{b,q}( \return\circ\codingmap )}
    =-\lambda(\mu_{b,q})
    \text{ and }
    \frac{\partial}{\partial q}p(b,q)=\frac{-\codingmeasure_{b,q}( \indupote\circ\codingmap )}{\codingmeasure_{b,q}( \return\circ\codingmap )}
    =-\mu_{b,q}( \phi ).
\end{align*}
Also, by the implicit function theorem and Ruelle's formula for the second derivative of the pressure function  \cite[Proposition 2.6.14]{mauldin2003graph}, we obtain
\begin{align}\label{eq second derivatibe 1}
    \frac{\partial^2}{\partial q^2}p(b,q)
    =&
    \Bigg(
    \sigma^2_{b,q}(\indupote)\left(\codingmeasure_{b,q}( \return\circ\codingmap )\right)^2
    -2\sigma_{b,q}^2
    (\indupote,\return)
    \codingmeasure_{b,q}(\indupote\circ\codingmap )
    \codingmeasure_{b,q}( \return\circ\codingmap )
    \\&+\sigma^2_{b,q}(\return)\left(\codingmeasure_{b,q}(\indupote\circ\codingmap )\right)^2\Bigg)\frac{1}
    {(\codingmeasure_{b,q}( \return \circ \codingmap) )^3}.\nonumber
\end{align}
By \eqref{eq Abramov-Kac's formula}, we have
\begin{align*}
&    \sigma^2_{b,q}(\indupote)\left(\codingmeasure_{b,q}( \return\circ\codingmap )\right)^2
    -2\sigma_{b,q}^2
    (\indupote,\return)
    \codingmeasure_{b,q}( 
    \indupote\circ\codingmap)\codingmeasure_{b,q}
    ( \return\circ\codingmap )
    +\sigma^2_{b,q}(\return)\left(\codingmeasure_{b,q}(\indupote\circ\codingmap )\right)^2
    \\&
    =\lim_{n\to\infty}\frac{1}{n}
    \codingmeasure_{b,q}
    \Bigg(\left(
    S_n\left(\indupote\circ\codingmap-\codingmeasure_{b,q}(\indupote\circ\codingmap )\right)\codingmeasure_{b,q}( \return\circ\codingmap )
    \right.\\&\left.\left.
    -S_n\left(\return\circ\codingmap-\codingmeasure_{b,q}(\return\circ\codingmap )\right)\codingmeasure_{b,q}(\indupote\circ\codingmap )
    \right)^2\right)
    =(\codingmeasure_{b,q}(\return\circ\codingmap))^2
    \sigma^2_{b,q}(\indupote-\mu_{b,q}(\phi)\return).
\end{align*}
Combining this with \eqref{eq second derivatibe 1}, we obtain
\begin{align}\label{eq important formula second derivative}
    \frac{\partial^2}{\partial q^2}p(b,q)=\frac{\sigma^2_{b,q}(\indupote-\mu_{b,q}(\phi)\return)}{\measure_{b,q}( \return ) }.
\end{align}
We shall show the last statement in this theorem. If $\alpha_{\text{min}}=\alpha_{\text{max}}$ then there exists a constant $c\in\mathbb{R}$ such that $\{c\}=\{\int \phi d\mu:\mu\in M(f)\}$. Therefore, by \eqref{eq ruell's formula noninduced}, $\frac{\partial^2}{\partial q^2}p(b,q)=0$. 
Conversely, we assume that  $\frac{\partial^2}{\partial q^2}p(b,q)=0$. Then, by \eqref{eq important formula second derivative}, we have $\sigma^2_{b,q}(\indupote-\mu_{b,q}(\phi)\return)=0$. Thus, by \eqref{eq Abramov-Kac's formula}, Lemma \ref{lem finite pressure on a special parameter} and \cite[Lemma 4.8.8]{mauldin2003graph}, there exists bounded continuous function $\tilde u':\inducoding\rightarrow \mathbb{R}$ such that $(\indupote-\mu_{b,q}(\phi)\return)\circ\codingmap=\tilde u'-\tilde u'\circ\shift$.
Recall that, by Remark \ref{rem measurable bijection}, $\codingmap|_{\inducoding\setminus\codingmap^{-1}(J_0)}$ is one-to-one and $\codingmap|_{\inducoding\setminus\codingmap^{-1}(J_0)}^{-1}$ is continuous.
For 
$x\in J_0$ we fix $\tau_x\in \inducoding$ with $x=\codingmap(\tau_x)$ and define $\tilde u:\indulimit\rightarrow \mathbb{R}$ by $\tilde u|_{\indulimit\setminus J_0}=\tilde u'\circ\codingmap|_{\inducoding\setminus\codingmap^{-1}(J_0)}^{-1}$ and $\tilde u(x)=\tilde u'(\tau_x)$ for $x\in J_0$. Since $J_0$ is a countable set, $\tilde u$ is a Borel measurable bounded function satisfying 
\begin{align}\label{eq proof of cohomologous indu coboundary}
    \indupote-\mu_{b,q}(\phi)\return=\tilde u-\tilde u\circ\indumap. 
\end{align}
From this, we can see that for all $i\in\p$ we have 
\begin{align}\label{eq proof of cohomologous boundary}
\alpha_i=\mu_{b,q}(\phi).    
\end{align}
The proof of \eqref{eq proof of cohomologous boundary} proceeds as follows:
For a contradiction we assume that there exists $i\in\p$ such that $\alpha_i=\phi(x_i)<\mu_{b,q}(\phi)$. Then, since $\phi$ is continuous on $\Lambda$, there exist $N\in\mathbb{N}$ and a small number $\epsilon>0$ such that for all $n\geq N$ and $x\in I_{i^n}$ we have 
$\phi(x)<\mu_{b,q}(\phi)-\epsilon.$    
 We fix $\omega_1,\omega_2\in \edge\setminus\{i\}$ and a sequence $\{x_k\}_{k\in\mathbb{N}}$ with $x_k\in \codingmap([\omega_1 i^k\omega_2])$ ($k\in\mathbb{N}$). Then, we obtain
\[
\limsup_{k\to\infty}(\indupote-\mu_{b,q}(\phi)\return)(x_k)=\limsup_{k\to\infty}\sum_{j=0}^{k-1}(\phi-\mu_{b,q}(\phi))(f^j(x_k))=-\infty.
\]
However, since $\tilde u$ is a bounded function, for all $k\in\mathbb{N}$ we obtain $|(\indupote-\mu_{b,q}(\phi)\return)(x_k)|\leq 2\sup\{|u(x)|:{x\in \indulimit}\}<\infty$. This is a contradiction. By a similar argument, we also obtain a contradiction in the case $\alpha_i>\mu_{b,q}(\phi)$. Hence, we obtain \eqref{eq proof of cohomologous boundary}.
We set 
\[
N:=\bigcup_{i\in\p}\{x_i\},\ 
Z:=\bigcup_{i\in\p}\bigcup_{n\in\mathbb{N}}f^{-n}(x_i)\setminus N
\text{ and }
P:=\bigcup_{i\in\p}I_{ii}\setminus(N\cup Z\cup\indulimit).
\]
We will inductively construct a Borel measurable function $u:\Lambda\rightarrow \mathbb{R}$ such that for all $x\in \Lambda\setminus Z$ we have $\phi(x)=u(x)-u(f(x))+\mu_{b,q}(\phi)$.
Note that we have the following direct decomposition of $\Lambda$:
\begin{align}\label{eq direct sum of lambda}
    \Lambda=\indulimit\cup P\cup N\cup Z.
\end{align}
Define 
\begin{align}\label{eq proof of cohomologous construction 1}
    u(x):=\tilde u(x)
    \text{ for all $x\in \indulimit$}
    \text{ and }u(x)=0
    \text{ for all $x\in N\cup Z$}.
\end{align}
For $i\in\p$ we define
\[
P_{i,2}:=\bigcup_{\omega\in \edge\setminus\{i\}}I_{i^2\omega}\setminus(N\cup Z\cup \indulimit)
\text{ and }
P_{i,k}:=\bigcup_{\omega\in \edge\setminus\{i\}}I_{i^k\omega}\setminus(N\cup Z\cup P_{i,k-1})
\text{ for $k\geq 3$}
.
\]
Then, we obtain the direct decomposition $P=\bigcup_{i\in\p}\bigcup_{k\in\mathbb{N}}P_{i,k}$.
Let $i\in\p$.
Since for $x\in P_{i,2}$
we have $f(x)\in \indulimit$, $u(f(x))$ is already defined by \eqref{eq proof of cohomologous construction 1}. Thus,  the following definition is well-defined:
\begin{align}\label{eq proof of cohomologous construction 2}
u(x):=u(f(x))+\phi(x)-\mu_{b,q}(\phi)
\text{ for } x\in P_{i,2}.    
\end{align}
Next, we consider the definition $u(x)$ for $x\in P_{i,3}$. Since for $x\in P_{i,3}$ we have $f(x)\in P_{i,2}$, $u(f(x))$ is already defined by \eqref{eq proof of cohomologous construction 2}. For $x\in P_{i,3}$ we define
$u(x):=u(f(x))+\phi(x)-\mu_{b,q}(\phi)$.
Let $k\geq 4$.
Assume that for all $2\leq \ell\leq k$ and $x\in P_{i,\ell}$
we have already defined $u(x)$ by 
\[
u(x):=u(f(x))+\phi(x)-\mu_{b,q}(\phi).
\]
Since for $x\in P_{i,k+1}$
we have $f(x)\in P_{i,k}$, $u(f(x))$ is already defined. For $x\in P_{i,k+1}$ we define $u(x):=u(f(x))+\phi(x)-\mu_{b,q}(\phi)$. Therefore, by induction, the following definition is well-defined:
\begin{align}\label{eq proof of cohomologous def u on P}
    u(x)=u(f(x))+\phi(x)-\mu_{b,q}(\phi) \text{ for }x\in P.
\end{align}
We shall show that $u:\Lambda\rightarrow\mathbb{R}$ defined by \eqref{eq proof of cohomologous construction 1} and \eqref{eq proof of cohomologous def u on P} satisfies 
\begin{align}\label{eq proof of cohomologous coboundary}
    \phi(x)=u(x)-u(f(x))+\mu_{b,q}(\phi) \text{ for all } x\in \Lambda\setminus Z.
\end{align} 
By \eqref{eq proof of cohomologous boundary}, \eqref{eq proof of cohomologous construction 1} and \eqref{eq proof of cohomologous def u on P}, for $x\in P\cup N$ we have \eqref{eq proof of cohomologous coboundary}. By \eqref{eq proof of cohomologous indu coboundary}, for $x\in \codingmap(\bigcup_{\omega\in E_1}[\omega])$ we have \eqref{eq proof of cohomologous coboundary}. Thus, it is enough to show that for all $n\geq 2$ and $x\in \codingmap(\bigcup_{\omega\in E_n}[\omega])$ we have \eqref{eq proof of cohomologous coboundary}. Let $n\geq 2$ and let $x\in \codingmap(\bigcup_{\omega\in E_n}[\omega])\setminus\codingmap(\bigcup_{\omega\in E_{n-1}}[\omega])$.
Note that, by \eqref{eq foem edges 2}, there exists $i\in\p$ such that for all $1\leq k\leq n-1$ we have $f^{k}(x)\in P_{i,n-(k-1)}$.
By \eqref{eq proof of cohomologous indu coboundary}, \eqref{eq proof of cohomologous construction 1} and \eqref{eq proof of cohomologous def u on P}, we have 
\begin{align*}
    &\sum_{k=0}^{n-1}(\phi-\mu_{b,q}(\phi))(f^k(x))=\tilde u (x)-\tilde u(f^n(x))
    \\&= u(x)-u(f(x))+\sum_{k=1}^{n-1}\left(u(f^k(x))-u(f^{k+1}(x))\right)
    \\&=u(x)-u(f(x))+\sum_{k=1}^{n-1}(\phi-\mu_{b,q}(\phi))(f^k(x)).
\end{align*}
Hence, we obtain $\phi(x)=u(x)-u(f(x))+\mu_{b,q}(\phi)$. This completes the proof of \eqref{eq proof of cohomologous coboundary}. 
It remains to show that for all $\mu\in M(f)$ we have $\mu(\phi)=\mu_{b,q}(\phi)$. Let $\mu\in M(f)$. Since $Z$ is countable set, there is no periodic orbits in $Z$ and $\mu$ is $f$-invariant, we have $\mu(Z)=0$. Thus, by \eqref{eq proof of cohomologous coboundary}, for all $\mu\in M(f)$ we obtain
\[
\mu(\phi)=\int_{\Lambda\setminus Z}\phi d\mu
=\int_{\Lambda\setminus Z}(u-u\circ f)d\mu+\mu_{b,q}(\phi)=\mu_{b,q}(\phi).
\]
This implies that $\alpha_{\text{min}}=\alpha_{\text{max}}=\mu_{b,q}(\phi)$ and the proof is complete.
\end{proof}

For $(b,q)\in\mathbb{R}^2$ we define the set of equilibrium measures for $-q\phi-b\log |f'|$ by 
\[
M_{b,q}:=\{\nu\in M(f):p(b,q)=h(\nu)+\nu(-q\phi-b\log |f'|)\}.
\]
For a convex function $(x_1,\cdots,x_n)\in\mathbb{R}^n\mapsto V(x_1,\cdots,x_n)\in\mathbb{R}$ $(n\in\mathbb{N})$,  $\boldsymbol{\hat x}=(\hat x_1,\cdots,\hat x_n)\in\mathbb{R}^2$ and $1\leq k\leq n$ we denote by $V^+_{x_k}(\boldsymbol{\hat x})$ the right-hand derivative of $V$ with respect to the variable $x_k$ at $\boldsymbol{\hat x}$ and by $V^-_{x_k}(\boldsymbol{\hat x})$ the left-hand derivative of $V$ with respect to the variable $x_k$ at $\boldsymbol{\hat x}$.
 \begin{prop}\label{prop derivatibe of pressure without conditions}
     For all $(b_0,q_0)\in \mathbb{R}^2$ we have 
     \begin{align*}
         p^+_{q}(b_0,q_0)=\sup_{\nu\in M_{b_0,q_0}}\{-\nu(\phi)\} \text{ and }
         p^-_{q}(b_0,q_0)=\inf_{\nu\in M_{b_0,q_0}}\{-\nu(\phi)\}.
     \end{align*}
      \end{prop}
      \begin{proof}
This follows from a slight modification of the proof of \cite[Proposition 2.3]{kessebohmer2004multifractaltobeappdated}, using the compactness of $M(f)$ in the weak* topology, the convexity of the topological pressure, and the variational principle for the topological pressure.
      \end{proof}

\section{Multifractal analysis}\label{sec multifractal}
In this section, we prove Theorem \ref{thm main}.
As in the previous section, we assume that $f\in \maps$ and $\phi\in \mathcal{R}$ in this section.
Recall that 
$\dimension:=\dim_H(\Lambda).$
Since each branch $f_i$ ($i\in\edge$) of $f$ is $C^{1+\varepsilon}$ and the condition (F2) holds, we can apply \cite[Theorem 4.6]{mauldin2000parabolictobeappdated} to obtain the following theorem:
\begin{thm}\label{thm bowen formula}{\cite[Theorem 4.6]{mauldin2000parabolictobeappdated}}
    We have
    \[
    \delta=\sup\left\{\frac{h(\nu)}{\lambda(\nu)}:\nu\in M(f),\ \lambda(\nu)>0\right\}
    =\min\{t\in\mathbb{R}: P(-t\log|f'|)\leq 0\}.
    \]
\end{thm}
By Theorem \ref{thm bowen formula}, we can apply the results of \cite{JJTOPnonuniformly}.
\begin{thm}\label{thm coditional variational principle}
    For all $\alpha\in (\alpha_{\text{min}},\alpha_{\text{max}})\setminus [\underline{\alpha}_\p,\overline{\alpha}_\p]$ we have 
    \[
    b(\alpha)=\sup\left\{\frac{h(\nu)}{\lambda(\nu)}:\nu\in M(f),\ \lambda(\nu)>0,\  \nu(\phi)=\alpha\right\}.
    \]
    Moreover, for all $\alpha\in [\underline{\alpha}_\p,\overline{\alpha}_\p]$ we have $b(\alpha)=\delta$. 
\end{thm}

\begin{lemma}\label{lemma indu and nonindu limit}
   We have $\delta=\dim_H(\indulimit)$. 
\end{lemma}
\begin{proof}
    We first note that
    $\Lambda
    =\indulimit\cup\bigcup_{i\in\p}I_{ii}
    $
    \text{ and }
    $
     \bigcup_{i\in\p}\bigcup_{n=2}^{\infty} T_i^n(\indulimit)=\bigcup_{i\in\p}I_{ii}\setminus \bigcup_{i\in\p}\bigcup_{k=0}^\infty f^{-k}(x_i).
    $
Therefore, since for each $i\in\p$ and $n\in\mathbb{N}$ the map $T_i^n$ is bi-Lipschitz,  the set $\Lambda\setminus   \bigcup_{i\in\p}\bigcup_{k=0}^\infty f^{-k}(x_i)$ is a countable union of bi-Lipschitz images of $\indulimit$. Thus, by \cite[Section 3]{falconer2013fractaltobeappdated},  $\dim_H(\Lambda\setminus   \bigcup_{i\in\p}\bigcup_{k=0}^\infty f^{-k}(x_i))=\dim_H(\indulimit)$. Since $\bigcup_{i\in\p}\bigcup_{k=0}^\infty f^{-k}(x_i)$ is a countable set, we are done.   
    \end{proof}

\begin{thm}\label{thm maximal measure and liftable problem}
We have $\tilde P(-\delta\log |\indumap'\circ\codingmap|)=0$. Moreover, $\codingmeasure_{\delta,0}(\return\circ\codingmap)<\infty$  if and only if $\exponent>2/\delta-1$.  
\end{thm}
\begin{proof}
    By (F1), 
    for all $b\in \mathbb{R}$ we have 
\begin{align}\label{eq proof of thm maximal measure}
\sum_{\omega\in \edgesindu}e^{-b\log|\indumap'\circ\codingmap|([\omega])}
 \asymp\sum_{n=1}^{\infty}\frac{1}{n^{b(1+\exponent)}}.     
\end{align}
 Therefore, since $\edgesindu$ is finitely primitive and $\log|\indumap'\circ\codingmap|$ is locally H\"older, we have $\lim_{b\to (1+\exponent)^{-1}}\tilde P(-b\log |\indumap'\circ\codingmap|)=\infty$. By Bowen's formula \cite[Theorem 4.2.13]{mauldin2000parabolictobeappdated} and Lemma \ref{lemma indu and nonindu limit}, this yields that $\tilde P(-\delta\log |\indumap'\circ\codingmap|)=0$. Moreover, by (F1) and  \eqref{eq gibbs}, we have 
\begin{align*}
    \codingmeasure_{\delta,0}(\return\circ\codingmap)=\sum_{n\in\mathbb{N}}\sum_{\omega\in \edgepart_n}n\codingmeasure_{\delta,0}([\omega])\asymp
    \sum_{n\in\mathbb{N}}\sum_{\omega\in \edgepart_n}ne^{-\delta\log|\indumap'\circ\codingmap|([\omega])}
    \asymp\sum_{n=1}^{\infty}\frac{n}{n^{\delta(1+\exponent)}}.
\end{align*}
Therefore, we obtain the last statement and the proof is complete.
\end{proof}


We define $\codingmeasure_{\delta}:=\codingmeasure_{\delta,0}$ and  $\measure_{\delta}=\codingmeasure_\delta\circ\codingmap^{-1}$.
If $\exponent>2/\delta-1$ then we define
\begin{align}
    \mu_{\delta}:=\frac{1}{\measure_{\delta}(\return )}\sum_{n=0}^{\infty}\sum_{k=n+1}^\infty\measure_{\delta}|_{\{\return=k\}}\circ f^{-n}.
\end{align}
We also define
$M_{\delta}:=M_{\delta,0}=\{\nu\in M(f):p(b,0)=h(\nu)+\nu(-\delta\log |f'|)\}.
$

\begin{prop}\label{prop equilibrium state Lambda}
    We have $M_{\delta}=\conv$ if  $\exponent\leq2/\delta-1$ and $M_{\delta}=\text{Conv}(\{\delta_{x_i}\}_{i\in\p},\mu_{\delta})$ if $\exponent>2/\delta-1$. 
\end{prop}

\begin{proof}
By Theorem \ref{thm bowen formula}, we have $\bigcup_{i\in\p}\{\delta_{x_i}\}\subset M_{\delta}$.
Let $\nu\in M_{\delta}$ be an ergodic measure such that $\nu\notin \conv$. Then, by Lemma \ref{lemma equivalent condition not liftable}, we have $\nu(\indulimit)>0$. 
Let $\tilde\nu=\nu|_{\indulimit}/\nu(\indulimit)$.
By Remark \ref{rem measurable bijection}, there exists $\tilde \nu'\in M(\shift)$ such that $\tilde \nu=\tilde \nu'\circ\codingmap^{-1}$ and $h(\tilde \nu)=h(\tilde \nu')$.
Then, by Theorem \ref{thm variational principle induce}, Theorem \ref{thm maximal measure and liftable problem} and \eqref{eq classical Abramov-Kac's formula}, we obtain
\[
0=\tilde P(-\delta\log|\indumap'\circ\codingmap|)\geq h(\tilde \nu')-\delta\lambda(\tilde \nu')={\tilde \nu(\return)}(h(\nu)-\delta\lambda(\mu))=0,
\]
Therefore, $\tilde P(-\delta\log|\indumap'\circ\codingmap|)=h(\tilde \nu')-\delta\lambda(\tilde \nu')$ and $\tilde\nu'$ is an equilibrium measure for $\log |\indumap'\circ\codingmap|$. By the uniqueness of the equilibrium measure for $\log |\indumap'\circ\codingmap|$ (see Theorem \ref{thm uniquness and existence of the equilibrium state}), we obtain 
\begin{align}\label{eq proof equilibrium state geometric}
    \tilde \nu'=\codingmeasure_{\delta}.
\end{align}
We first assume that $\exponent\leq 2/\delta-1$. Then, by Theorem \ref{thm maximal measure and liftable problem} and \eqref{eq classical Abramov-Kac's formula}, we have $\infty=\codingmeasure_{\delta}(\return\circ\codingmap)=\tilde\nu'(\return\circ\codingmap)=1/\nu(\indulimit)<\infty$. This is a contradiction. Therefore, the set of ergodic measures in $M_\delta$ is $\bigcup_{i\in\p}\{\delta_{x_i}\}$. By the ergodic decomposition theorem (see \cite[Theorem 5.1.3]{viana}), $M_\delta=\conv$.
Next, we consider the case $\exponent>2/\delta-1$. In this case, by \eqref{eq proof equilibrium state geometric}, the set of ergodic measures in $M_\delta$ is $\bigcup_{i\in\p}\{\delta_{x_i}\}\cup \{\mu_\delta\}$. Thus, by ergodic decomposition theorem, we obtain $M_{\delta}=\text{Conv}(\{\delta_{x_i}\}_{i\in\p},\mu_{\delta})$. 
\end{proof}

By Theorem \ref{thm bowen formula}, for all $b\in (-\infty,\delta)$ we have $p(b,0)>0$ and thus, $(b,0)\in \mathcal{N}$. For $b\in(-\infty,\delta)$ we set $\codingmeasure_{b}:=\codingmeasure_{b,0}$, $\measure_b=\measure_{b,0}$ and $\mu_{b}:=\mu_{b,0}$. We recall that $\frat$ is defined by \eqref{eq def frat}.

\begin{lemma}\label{lemma equilibrium measure near Delta}
There exists $\{b_n\}_{n\in\mathbb{N}}\subset (-\infty,\delta)$ such that $\lim_{n\to\infty}b_n=\delta$ and   $\lim_{n\to\infty}\mu_{b_n}(\phi)\in \frat$.
    \end{lemma}

\begin{proof}
    We first note that, since $\edge$ is finite, $\Lambda$ is compact and $M(f)$ is also compact with respect to the weak* topology (see \cite[Theorem 6.10]{walters2000introduction}). Let $\{b_n\}_{n\in\mathbb{N}}\subset(-\infty,\delta)$ be a sequence such that $\lim_{n\to\infty}b_n=\delta$. Since $M(f)$ is compact, there exist a subsequence $\{b_{n_k}\}_{k\in\mathbb{N}}\subset \{b_n\}_{n\in\mathbb{N}}$ and $\mu\in M(f)$ such that $\mu_{b_{n_{k}}}$ converges to $\mu$ as $k\to\infty$ in the weak* topology. We shall show that $\mu\in M_{\delta}$. Since the entropy map is upper semi-continuous, $\mu_{b_{n_{k}}}$ ($k\in\mathbb{N}$) is the equilibrium measure for $-b_{n_{k}}\log |f'|$ and the map $b\mapsto p(b,0)$ is continuous on $\mathbb{R}$, we have 
    \[
    h(\mu)-\delta\lambda(\mu)\geq \limsup_{k\to\infty}( h(\mu_{b_{n_k}})-b_{n_k}\lambda(\mu_{b_{n_k}}))=\limsup_{k\to\infty}p(b_{n_k},0)=p(\delta,0).
    \]
Thus, $\mu\in M_\delta$. By Proposition \ref{prop equilibrium state Lambda}, we obtain $\lim_{k\to\infty}\mu_{b_{n_k}}(\phi)=\mu(\phi)\in \frat$. 
\end{proof}

For $\alpha\in (\alpha_{\text{min}},\alpha_{\text{max}})$ we define the function $\palpha:\mathbb{R}^2\rightarrow\mathbb{R}$ by 
\[
\palpha(b,q):=p(b,q)+q\alpha=P(q(-\phi+\alpha)-b\log|f'|).
\]
We denote by $\underline{i}$ the index in $\p$ satisfying $\alpha_{\underline{i}}=\underline{\alpha}_\p$ and by $\overline{i}$ the index in $\p$ satisfying $\alpha_{\overline{i}}=\overline{\alpha}_\p$, where $\underline{\alpha}_\p:=\min_{i\in\p}\{\alpha_i\}$ and $\overline{\alpha}_\p:=\max_{i\in\p}\{\alpha_i\}$.

\begin{lemma}\label{lemma positivity p alpha}
    For all $\alpha\in (\alpha_{\text{min}},\alpha_{\text{max}})$ and $q\in \mathbb{R}$ we have $\palpha(b(\alpha),q)\geq 0$. Moreover, for all $b\in \mathbb{R}$ we have 
    \begin{align}\label{eq asymptotic behavior of alpha pressure}
        \lim_{|q|\to\infty}\palpha(b,q)=\infty.
    \end{align}
\end{lemma}
\begin{proof}
Let $\alpha\in (\alpha_{\text{min}},\alpha_{\text{max}})$ and let $q\in \mathbb{R}$. 
    If $\alpha\in [\underline{\alpha}_\p,\overline{\alpha}_\p]=[\alpha_{\underline{i}},\alpha_{\overline{i}}]$ then there exists $p\in[0,1]$ such that we have $\alpha=p\alpha_{\underline{i}}+(1-p)\alpha_{\overline{i}}$. Then, by using the measure $p\delta_{x_{\underline{i}}}+(1-p)\delta_{x_{\overline{i}}}$, we can verify that $\palpha(b(\alpha),q)\geq 0$. If $\alpha\in (\alpha_{\text{min}},\alpha_{\text{max}})\setminus [\underline{\alpha}_\p,\overline{\alpha}_\p]$ then by Theorem \ref{thm coditional variational principle}, there exists $\{\mu_n\}_{n\in\mathbb{N}}\subset M(f)$ such that for all $n\in\mathbb{N}$ we have $\lambda(\mu_n)>0$, $\mu_n(\phi)=\alpha$ and $\lim_{n\to\infty}h(\mu_n)/\lambda(\mu_n)=b(\alpha)$. Thus, 
    $
    \palpha(b(\alpha),q)\geq h(\mu_n)-b(\alpha)\lambda(\mu_n)$ and  letting $n\to\infty$, we obtain $\palpha(b,q)\ge 0$.  

    Next, we shall show \eqref{eq asymptotic behavior of alpha pressure}. Let $\alpha\in (\alpha_{\text{min}},\alpha_{\text{max}})$ and let $b\in\mathbb{R}$. Since $\alpha \in (\alpha_{\text{min}},\alpha_{\text{max}})$, there exists $\underline{\nu},\overline{\nu}\in M(f)$ such that $\underline{\nu}(\phi)<\alpha<\overline{\nu}(\phi)$. Hence, we obtain
     $   \lim_{q\to\infty}\palpha(b,q)
        \geq 
         h(\overline{\nu})+\lim_{q\to\infty}q(\overline{\nu}(\phi)-\alpha)-b\lambda(\overline{\nu})=\infty
        $\text{ and }
        $
        \lim_{q\to-\infty}\palpha(b,q)
        \geq 
         h(\underline{\nu})+\lim_{q\to-\infty}q(\underline{\nu}
         (\phi)-\alpha)-b\lambda(\underline{\nu})=\infty.
         $
\end{proof}

\begin{prop}\label{prop only frat}
     For all $\alpha\in (\alpha_{\text{min}},\alpha_{\text{max}})\setminus \frat$ we have $b(\alpha)<\delta$.
\end{prop}

\begin{proof}
    We first consider the case $\alpha\in (\alpha_{\text{min}},\min \frat)$. For a contradiction, we assume that there exists  $\alpha\in (\alpha_{\text{min}},\min \frat)$ such that $b(\alpha)=\delta$. By Theorem \ref{thm bowen formula} and Lemma \ref{lemma positivity p alpha}, we have $\palpha(\delta,0)=0$ and $\palpha(\delta,q)\geq 0$ for all $q\in(0,\infty)$. By the convexity of the function $q\mapsto\palpha(\delta,q)$, we obtain $(\palpha)_q^+(\delta,0)\geq0$. However, by Proposition \ref{prop derivatibe of pressure without conditions} and Proposition \ref{prop equilibrium state Lambda}, we have
    \begin{align*}
        (\palpha)_q^+(\delta,0)=\sup_{\nu\in M_\delta}\{-\nu(\phi)\}+\alpha= -\min \frat+\alpha<0.
    \end{align*}
    This is a contradiction. Therefore, for all $\alpha\in (\alpha_{\text{min}},\min \frat)$ we have $b(\alpha)<\delta$. By a similar argument, one can show that  for all $\alpha\in (\max \frat,\alpha_{\text{max}})$ we have $b(\alpha)<\delta$. 
    \end{proof}

\begin{prop}\label{prop relationship}
    For each $\alpha\in (\alpha_{\text{min}},\min \frat)$ (resp. $\alpha\in (\max\frat,\alpha_{\text{max}})$) there exists a unique number $q(\alpha)\in (0,\infty)$ (resp. $q(\alpha)\in (-\infty,0)$) such that $(b(\alpha),q(\alpha))\in \mathcal{N}$ and  
    \begin{align}\label{eq relation in the statement}
\palpha(b(\alpha),q(\alpha))=0 \text{ and } \frac{\partial}{\partial q}\palpha(b(\alpha),q(\alpha))=0.        
    \end{align}
    Moreover, $\mu_{b(\alpha),q(\alpha)}$ is a unique measure $\nu\in M(f)$ satisfying $\lambda(\nu)>0$, $\nu(\phi)=\alpha$ and $b(\alpha)=h(\nu)/\lambda(\nu)$.
\end{prop}
\begin{proof}
If $\alpha_{\text{min}}=\alpha_{\text{max}}$, there is nothing to prove. Thus, we assume that $\alpha_{\text{min}}<\alpha_{\text{max}}$.
Let $\alpha\in(\alpha_{\text{min}},\min \frat)$. Then, by Theorem \ref{thm bowen formula}, Lemma \ref {lemma positivity p alpha} and Proposition \ref{prop only frat}, we have $p(b(\alpha),0)>0$ and $p(b(\alpha),q)\geq -q\alpha>-q\min A\geq \max_{i\in\p}\{-q\alpha_i\}$ for all $q\in (0,\infty)$. Hence, $\{b(\alpha)\}\times[0,\infty)\in \mathcal{N}$. 
On the other hand, by Theorem \ref{thm bowen formula}, for all $b\in (-\infty,\delta)$ we have $\palpha(b,0)>0$ and thus, $(-\infty,\delta)\times\{0\}\in \mathcal{N}$. 
Hence, by Theorem \ref{thm regularity of non induced pressure}, the function $\palpha$ is real-analytic on a open set $\mathcal{O}$ containing $\{b(\alpha)\}\times [0,\infty)\cup (-\infty,\delta)\times \{0\}$. We will show that 
\begin{align}\label{eq proof of relation}
\frac{\partial}{\partial q}\palpha(b(\alpha),0)<0.    
\end{align}
For a contradiction, we assume that 
\begin{align}\label{eq proof of relation contradiction}
\frac{\partial}{\partial q}\palpha(b(\alpha),0)\geq 0.    
\end{align}
We take a small number $\epsilon>0$ such that  $\alpha<\min \frat-\epsilon$. By Lemma \ref{lemma equilibrium measure near Delta}, there exists $b_0\in(b(\alpha),\delta)$ such that  $\min\frat-\epsilon\leq \mu_{b_0}(\phi)$. Then, by Theorem \ref{thm regularity of non induced pressure}, we have 
\begin{align*}
    \frac{\partial}{\partial q}\palpha(b_0,0)=-\mu_{b_0}(\phi)+\alpha<-\min A+\epsilon+\alpha<0.
\end{align*}
Combining this with \eqref{eq proof of relation contradiction} and using the continuity of the function $b\mapsto \frac{\partial}{\partial q}\palpha(b,0)$, we conclude that there exists $b'\in [b(\alpha),b_0)$ such that $\frac{\partial}{\partial q}\palpha(b',0)=0$. This yields that $\mu_{b'}(\phi)=\alpha$. Thus, by Theorem \ref{thm coditional variational principle}, we obtain 
\[
0<\palpha(b',0)=h(\mu_{b'})-b'\lambda(\mu_{b'})\leq h(\mu_{b'})-b(\alpha)\lambda(\mu_{b'})\leq h(\mu_{b'})-\frac{h(\mu_{b'})}{\lambda(\mu_{b'})}\lambda(\mu_{b'})=0.
\]
This is a contradiction and we obtain \eqref{eq proof of relation}. Since we assume that $\alpha_{\text{min}}<\alpha_{\text{max}}$, Theorem \ref{thm regularity of non induced pressure} yields that the function $q\mapsto\palpha(b(\alpha),q)$ is strictly convex on $[0,\infty)$. Thus, by \eqref{eq proof of relation} and Lemma \ref{lemma positivity p alpha}, there exists a unique number $q(\alpha)\in (0,\infty)$ such that $\frac{\partial}{\partial q}\palpha(b(\alpha),q(\alpha))=0$. In particular $\mu_{b(\alpha),q(\alpha)}(\phi)=\alpha$. Moreover, by Theorem \ref{thm coditional variational principle} and Lemma \ref{lemma positivity p alpha}, we obtain
\[
0\leq \palpha(b(\alpha),q(\alpha))=h(\mu_{b(\alpha),q(\alpha)}) -b(\alpha)\lambda(\mu_{b(\alpha),q(\alpha)})\leq 0
\]
and thus, $\palpha(b(\alpha),q(\alpha))=0$ and $b(\alpha)=h(\mu_{b(\alpha),q(\alpha)})/\lambda(\mu_{b(\alpha),q(\alpha)})$. Hence, for all $\alpha\in(\alpha_{\text{min}},\min \frat)$, the proof of the first half is complete. By a similar argument, we can show that for all $\alpha\in (\max A, \alpha_{\text{max}})$  there exists a unique number $q(\alpha)\in (-\infty,0)$ such that $(b(\alpha),q(\alpha))\in \mathcal{N}$ and \eqref{eq relation in the statement} holds. Moreover, we have $b(\alpha)=h(\mu_{b(\alpha),q(\alpha)})/\lambda(\mu_{b(\alpha),q(\alpha)})$. 

Next, we shall show that the second half. Let $\alpha\in (\alpha_{\text{min}},\alpha_{\text{max}})\setminus A$ and let $\nu$ be in $M(f)$ such that $\lambda(\nu)>0$, $\nu(\phi)=\alpha$ and $b(\alpha)=h(\nu)/\lambda(\nu)$. Then, we have
\[
h(\nu)+q(\alpha)(-\nu(\phi)+\alpha)-b(\alpha)\lambda(\nu)=0=p(b(\alpha),q(\alpha))+q(\alpha)\alpha.
\]
Thus, $\nu$ is an equilibrium measure for $-q(\alpha)\phi-b(\alpha)\log|f'|$. Therefore, since $(b(\alpha),q(\alpha))\in \mathcal{N}$, the uniqueness of an equilibrium measure for $-q(\alpha)\phi-b(\alpha)\log|f'|$ (Theorem \ref{thm uniquness and existence of the equilibrium state}) yields that $\nu=\mu_{b(\alpha),q(\alpha)}$. 
\end{proof}

\begin{prop}\label{prop real analytic}
    The functions $\alpha\mapsto b(\alpha)$ and $\alpha\mapsto q(\alpha)$ are real-analytic on $(\alpha_{\text{min}},\alpha_{\text{max}})\setminus \frat$.
\end{prop}

\begin{proof}
If $\alpha_{\text{min}}=\alpha_{\text{max}}$, there is nothing to prove. Thus, we assume that $\alpha_{\text{min}}<\alpha_{\text{max}}$.
We proceed with this proof as in \cite[Lemma 9.2.4]{Barreirabook}.
We define the function $G:\mathbb{R}^3\rightarrow \mathbb{R}^2$ by $G(\alpha,b,q):= (\palpha(b,q), \frac{\partial}{\partial q}\palpha(b,q))=(p(b,q)+q\alpha, \frac{\partial}{\partial q}p(b,q)+\alpha)$. By Proposition \ref{prop relationship}, for all $\alpha\in (\alpha_{\text{min}},\alpha_{\text{max}})\setminus \frat$ we have $G(\alpha,b(\alpha),q(\alpha))=0$. 
We want to apply the implicit function theorem in order to show the regularity of the functions $b$ and $q$. To do this, it is sufficient to show that 
\begin{align*}
    \text{det}\begin{pmatrix}
   \frac{\partial}{\partial b}\palpha(b(\alpha),q(\alpha)) & \frac{\partial}{\partial q}\palpha(b(\alpha),q(\alpha)) \\
   \frac{\partial^2}{\partial b\partial q}\palpha(b(\alpha),q(\alpha)) & 
   \frac{\partial^2}{\partial q^2}\palpha(b(\alpha),q(\alpha))
\end{pmatrix}
\neq 0.
\end{align*}
Note that, by Proposition \ref{prop relationship} and Theorem \ref{thm regularity of non induced pressure}, we have $\frac{\partial}{\partial q}\palpha(b(\alpha),q(\alpha))=0$ and $\frac{\partial}{\partial b}\palpha(b(\alpha),q(\alpha))=\lambda(\mu_{b(\alpha),q(\alpha)})>0$.
By $\alpha_{\text{min}}<\alpha_{\text{max}}$, Theorem \ref{thm regularity of non induced pressure} implies that $\frac{\partial^2}{\partial q^2}\palpha(b(\alpha),q(\alpha))\neq 0$. Therefore, by the implicit function theorem and Theorem \ref{thm regularity of non induced pressure}, we are done. 
\end{proof}

\begin{prop}\label{prop monotone}
    The function $b$ is strictly increasing on $(\alpha_{\text{min}},\min \frat)$ and  strictly decreasing on $(\max \frat, \alpha_{\text{max}})$.
\end{prop}
\begin{proof}
    Let $\alpha_1,\alpha_2\in (\alpha_{\text{min}},\min \frat)$.
    We assume that $\alpha_1<\alpha_2$. By Proposition \ref{prop only frat}, we have $b(\alpha_1)<\delta$. 
    We take a small $\epsilon>0$ with $\alpha_2<\min\frat-\epsilon$ and $b(\alpha_1)<\delta-\epsilon$. We first show that there exists $\mu\in M(f)$ such that  
    \begin{align}\label{eq proof of monotonicity saturated}
        \lambda(\mu)>0,\ \min\frat-\epsilon<\mu(\phi) \text{ and } \delta-\epsilon<\frac{h(\mu)}{\lambda(\mu)}.
    \end{align}
    By Theorem \ref{thm bowen formula}, there exists $\nu\in M(f)$ such that $\lambda(\nu)>0$ and $\delta-\epsilon<h(\nu)/\lambda(\nu)$. Moreover, there exists $p\in [0,1)$ such that $\min A-\epsilon<p\alpha_{\underline{i}}+(1-p)\nu(\phi)$. We set $\mu:=p\delta_{x_{\underline{i}}}+(1-p)\nu$. Then, by the affinity of the entropy map (\cite[Theorem 8.1]{walters2000introduction}), we obtain $h(\mu)/\lambda(\mu)=h(\nu)/\lambda(\nu)$ and thus, $\mu$ satisfies \eqref{eq proof of monotonicity saturated}.
    
    Since $\alpha_1<\alpha_2<\mu(\phi)$, there exists $\bar p\in (0,1)$ such that $\alpha_2=\bar p\alpha_1+(1-\bar p)\mu(\phi)$.  
We set $\bar\nu=\bar p\mu_{b(\alpha_1),q(\alpha_1)}+(1-\bar p)\mu$. By Proposition \ref{prop relationship}, \eqref{eq proof of monotonicity saturated} and the affinity of the entropy map, we obtain $\bar\nu(\phi)=\alpha_2$ and 
\begin{align*}
&h(\bar \nu)=\bar ph(\mu_{b(\alpha_1),q(\alpha_1)})+(1-\bar p)h(\mu)
\\&>b(\alpha_1)(\bar p\lambda(\mu_{b(\alpha_1),q(\alpha_1)})+(1-\bar p)\lambda(\mu))    
=b(\alpha_1)\lambda(\bar \nu).
\end{align*}
Hence, by Theorem \ref{thm coditional variational principle}, we obtain $b(\alpha_1)<b(\alpha_2)$. This implies that $b$ is strictly increasing on $(\alpha_{\text{min}},\min A)$. By the similar argument, we can show that $b$ is strictly decreasing on $(\max A, \alpha_{\text{max}})$.
\end{proof}

\emph {Proof of Theorem $\ref{thm main}$.}
(1) of Theorem \ref{thm main} follows from Proposition \ref{prop relationship}. (2) follows from Proposition \ref{prop real analytic}, and (3) from Propositions \ref{prop only frat} and \ref{prop monotone}.
\qed

\subsection*{Acknowledgments}
This work was supported by the JSPS KAKENHI 25KJ1382.

\bibliographystyle{abbrv}
\bibliography{reference}
 \nocite{*}

\end{document}